\documentclass{article}
\usepackage{amsmath, amsthm, hyperref}

\theoremstyle{theorem}
\newtheorem{theorem}{Theorem}
\newtheorem{corollary}{Corollary}
\newcommand{\E}{\mathbf{E}}
\newcommand{\R}{\mathbf{R}}

\begin{document}

\title{A Probabilistic Proof of a Wallis-type Formula for the Gamma Function}
\markright{Notes}
\author{Wooyoung Chin}

\maketitle

\begin{abstract}
	We use well-known limit theorems in probability theory to derive a Wallis-type product formula for the gamma function.
	Our result immediately provides a probabilistic proof of Wallis's product formula for $\pi$, as well as the duplication formula for the gamma function.
\end{abstract}

In 1655, Wallis \cite[Prop. 191]{Wal56} wrote down the following beautiful formula for $\pi$:
\begin{equation} \label{eq:Wallis}
	\frac{\pi}{2}
	= \prod_{n=1}^{\infty} \left( \frac{2n}{2n-1} \cdot \frac{2n}{2n+1} \right)
	= \frac{2}{1} \cdot \frac{2}{3} \cdot \frac{4}{3} \cdot \frac{4}{5} \cdot \frac{6}{5} \cdot \frac{6}{7} \cdot \frac{8}{7} \cdots.
\end{equation}
Ever since the formula's discovery, various proofs of  {\it Wallis's product formula} have been found, and each of them has its own merits.
One of the more common proofs of the formula uses a recursion derived from integrating trigonometric functions.
Another proof simply plugs in $x = \pi/2$ into Euler's infinite product formula
\begin{equation}\label{eq:Euler}
	\frac{\sin x}{x} = \prod_{n=1}^{\infty} \left(1 - \frac{x^2}{n^2 \pi^2}\right).
\end{equation}
Although this proof is perhaps the shortest one, proving the above product formula for sine requires some amount of work.

The purpose of this note is to use well-known limit theorems in probability theory to derive a Wallis-type product formula for the gamma function.
A so-called duplication formula for the Gamma function will easily follow from the product formula.
The Gamma function $\Gamma:(0,\infty) \to \R$, which we only define for positive real numbers for simplicity, is given by
\begin{displaymath}
	\Gamma(\alpha) = \int_0^\infty e^{-t} t^{\alpha-1} \,dt.
\end{displaymath}
A direct computation shows $\Gamma(1/2) = \sqrt{\pi}$, and this will let us derive \eqref{eq:Wallis} from a more general product formula for the Gamma function.
By integration by parts, one can easily check that $\Gamma(\alpha) = (\alpha-1)\Gamma(\alpha-1)$ for any $\alpha > 1$.
From this $\Gamma(n) = (n-1)!$ for all $n \in \mathbf{N}$ follows.

The Gamma function is closely related to spheres and spherical coordinates.
For any $n \in \{2,3,4,\ldots\}$, the surface area of the unit sphere
\begin{displaymath}
	S^{n-1} = \{(x_1,\ldots,x_n) \in \R^n \mid x_1^2 + \cdots + x_n^2 = 1\}
\end{displaymath}
embedded in $\R^n$ is $2\pi^{n/2}/\Gamma(n/2)$.
Also, for any continuous $f:[0,\infty) \to \R$ with $\int_0^\infty \lvert f(r) \rvert r^{n-1}\,dr < \infty$, we have
\begin{equation} \label{eq:polar}
	\int_{\R^n} f\left(\sqrt{x_1^2 + \cdots + x_n^2}\right)\,dx_1 \cdots dx_n = \frac{2\pi^{n/2}}{\Gamma(n/2)} \int_0^\infty f(r) r^{n-1}\,dr.
\end{equation}
For more details on the Gamma function, see \cite[p. 58 and Section 2.7]{Fol99}.

If we restrict our interest to just proving \eqref{eq:Wallis}, then there already exist some probabilistic proofs.
A proof by Miller \cite{Mil08} uses the fact that for any $\nu \in \mathbf{N}$, the function $f:\R \to [0,\infty)$ given by
\begin{displaymath}
	f(t) = \frac{\Gamma\left(\frac{\nu+1}{2}\right)}{\sqrt{\pi\nu} \Gamma\left(\frac{\nu}{2}\right)}
	\left(1 + \frac{t^2}{\nu}\right)^{-\frac{\nu+1}{2}}
\end{displaymath}
is a probability density, i.e., it is nonnegative and has a total integral of one.
The distribution with the density $f$ is called Student's $t$-distribution with $\nu$ degrees of freedom.
Another proof by Wei, Li, and Zheng \cite{WLZ17} derives \eqref{eq:Wallis} from a version of the central limit theorem applied to certain familiar discrete random variables.

\begin{theorem} \label{conv}
	If $\alpha > 0$, then
	\begin{equation} \label{conv_even}
		\lim_{k\to\infty}k^{\alpha} \frac{(k-1) (k-2) \cdots 1}{(k-1+\alpha) (k-2+\alpha) \cdots \alpha}
		= \Gamma(\alpha)
	\end{equation}
	and
	\begin{equation} \label{conv_odd}
		\lim_{k\to\infty}\sqrt{\pi} k^{\alpha} \frac{\left(k-\frac{1}{2}\right) \left(k-\frac{3}{2}\right) \cdots \frac{1}{2} }{\left(k-\frac{1}{2}+\alpha\right) \left(k-\frac{3}{2}+\alpha\right) \cdots \left(\frac{1}{2} + \alpha\right)}
		=\Gamma\left(\alpha + \frac{1}{2}\right)
	\end{equation}
	where $k$ ranges over positive integers.
\end{theorem}

\begin{proof}
	Consider a family $X_1, X_2, \ldots$ of independent standard normal random variables.
	The proof is established by investigating the value of
	\begin{displaymath}
		\E \left(X_1^2 + \cdots + X_n^2\right)^{\alpha}
	\end{displaymath}
	in two different ways: the first way uses well-known limit theorems while the second way is purely computational.
	In fact, this value is the moment of order $\alpha$ of a chi-squared distribution.
	However, we won't assume any prior knowledge of chi-squared distributions in this note.
	
	Let us start with the approach using limit theorems.
	By the weak law of large numbers we have
	\begin{displaymath}
		\frac{X_1^2 + \cdots + X_n^2}{n} \to 1 \quad \textrm{in probability}.
	\end{displaymath}
	Applying the continuous function $f(y)=|y|^{\alpha}$ to both sides above and using the continuous mapping theorem \cite[Corollary 6.3.1 (ii)]{Res99}, which tells us that convergence in probability is preserved under continuous maps, we also get
	\begin{displaymath}
		\left(\frac{X_1^2 + \cdots + X_n^2}{n}\right)^{\alpha} \to 1 \quad \textrm{in probability}.
	\end{displaymath}
	
	Note that
	\begin{displaymath}
		\E \left(\frac{X_1^2 + \cdots + X_n^2}{n}\right)^2 = \frac{n \E X_1^4 + n(n-1) (\E X_1^2)^2}{n^2} \le \max\{\E X_1^4, (\E X_1^2)^2\}.
	\end{displaymath}
	Similarly, for any integer $p > \alpha$ we have
	\begin{displaymath}
		\E \left(\frac{X_1^2 + \cdots + X_n^2}{n}\right)^p \le \max\{\E X_1^{2p} , \ldots, (\E X_1^2 )^p\} < \infty.
	\end{displaymath}
	This shows that
	\begin{displaymath}
		\E \left( \left(\frac{X_1^2 + \cdots + X_n^2}{n}\right)^{\alpha} \right)^{p/\alpha}
		= \E \left(\frac{X_1^2 + \cdots + X_n^2}{n}\right)^p
	\end{displaymath}
	is bounded uniformly in $n$, and thus the family
	\begin{displaymath}
		\left\{ \left(\frac{X_1^2 + \cdots + X_n^2}{n}\right)^{\alpha} \right\}_{n = 1}^{\infty}
	\end{displaymath}
	is uniformly integrable. What we used here is sometimes called the ``crystal ball condition"; see \cite[p. 184]{Res99}.
	Since any uniformly integrable sequence of random variables that converges in probability also converges in $L^1$, see \cite[Theorem 6.6.1]{Res99}, we have
	\begin{equation} \label{conv_to_1}
		\lim_{n \to \infty} \E \left(\frac{X_1^2 + \cdots + X_n^2}{n}\right)^{\alpha} = \E 1 = 1.
	\end{equation}
	
	Let us next directly compute $\E (X_1^2 + \cdots + X_n^2)^{\alpha}$ by integration:
	\begin{align*}
		\E (X_1^2 + \cdots + X_n^2)^{\alpha}
		&= \int_{\R^n} (x_1^2 + \cdots + x_n^2)^{\alpha} \cdot \frac{1}{(2\pi)^{n/2}} e^{-(x_1^2 + \cdots + x_n^2)/2} \,dx_1 \cdots dx_n \\
		&= \frac{2 \cdot \pi^{n/2}}{\Gamma(n/2)}\int_0^\infty r^{2\alpha} \cdot \frac{1}{(2\pi)^{n/2}} e^{-r^2/2} \cdot r^{n-1} \, dr \\
		&= \frac{1}{2^{(n/2) - 1} \Gamma(n/2)} \int_0^\infty r^{n+2\alpha-1} e^{-r^2/2} \,dr .
	\end{align*}
	We used \eqref{eq:polar} in the second equality.
	Continuing our calculations, we observe that
	\begin{align*}
		&\frac{1}{2^{(n/2)-1} \Gamma(n/2)} \int_0^\infty r^{n+2\alpha-1} e^{-r^2/2} \,dr \\
		&= \frac{1}{2^{(n/2)-1} \Gamma(n/2)} \int_0^\infty (2u)^{(n/2)+\alpha-1} e^{-u} \,du \\
		&= 2^{\alpha} \cdot \frac{\Gamma\left((n/2)+\alpha\right)}{\Gamma(n/2)}.
	\end{align*}
	
	Finally, we conflate the two approaches.
	By \eqref{conv_to_1} and the previous computation, we have
	\begin{displaymath}
		\lim_{n\to\infty}\left(\frac{n}{2}\right)^{-\alpha} \cdot \frac{\Gamma\left((n/2)+\alpha\right)}{\Gamma(n/2)} =1.
	\end{displaymath}
	Plug in $n = 2k$ and $n=2k+1$. Then, using $\Gamma(x) = (x-1)\Gamma(x-1)$ to expand both the numerator and denominator of the left side, and applying $\lim_{k \to \infty} [(k+\frac{1}{2})/k]^\alpha = 1$, we have
	\begin{displaymath}
		\lim_{k \to \infty} k^{-\alpha} \frac{(k-1+\alpha) (k-2+\alpha) \cdots \alpha  \Gamma(\alpha)}{(k-1) (k-2) \cdots 1}
		= 1
	\end{displaymath}
	and
	\begin{displaymath}
		\lim_{k \to \infty} k^{-\alpha} \frac{\left(k-\frac{1}{2}+\alpha\right) \left(k-\frac{3}{2}+\alpha\right) \cdots \left(\alpha + \frac{1}{2}\right) \Gamma\left(\alpha + \frac{1}{2}\right)}{\left(k-\frac{1}{2}\right) \left(k-\frac{3}{2}\right) \cdots \frac{1}{2} \Gamma\left(\frac{1}{2}\right)}
		= 1.
	\end{displaymath}
	Taking the reciprocal and using $\Gamma(1/2) = \sqrt{\pi}$ concludes the proof.
\end{proof}

In case $\alpha$ is rational, we can estimate $\Gamma(\alpha)$ by a ratio of products of integers.

\begin{corollary} \label{gamma_formula}
	For any positive integers $p$ and $q$, we have
	\begin{displaymath}
		\lim_{k \to \infty} q\cdot \frac{k^{p/q}((k-1)q)((k-2)q) \cdots q}{((k-1)q+p)((k-2)q+p)\cdots p} = \Gamma\left(\frac{p}{q}\right),
	\end{displaymath}
	where $k$ ranges over positive integers.
\end{corollary}

\begin{proof}
	By applying \eqref{conv_even} with $\alpha = p/q$, we obtain
	\begin{multline*}
		\lim_{k \to \infty} q \cdot \frac{k^{p/q}((k-1)q)((k-2)q) \cdots q}{((k-1)q+p)((k-2)q+p)\cdots p} \\
		= \lim_{k \to \infty} k^{p/q} \cdot \frac{(k-1)(k-2) \cdots 1}{\left(k-1+\frac{p}{q} \right) \left(k-2+\frac{p}{q} \right) \cdots \frac{p}{q}}
		= \Gamma\left(\frac{p}{q}\right) .
	\end{multline*}
\end{proof}

The formula for $\Gamma(1/2)$ leads us to Wallis's original formula.

\begin{corollary}[Wallis]
	\begin{displaymath}
		\frac{\pi}{2}
		= \prod_{n=1}^{\infty} \left( \frac{2n}{2n-1} \cdot \frac{2n}{2n+1} \right)
		= \frac{2}{1} \cdot \frac{2}{3} \cdot \frac{4}{3} \cdot \frac{4}{5} \cdot \frac{6}{5} \cdot \frac{6}{7} \cdot \frac{8}{7} \cdots.
	\end{displaymath}
\end{corollary}

\begin{proof}
	Applying Corollary \ref{gamma_formula} with $p=1$ and $q=2$ gives
	\begin{displaymath}
		\lim_{k \to \infty} 2 \cdot \frac{\sqrt{k} \cdot 2 \cdot 4 \cdots (2k-4) (2k-2)}{1 \cdot 3 \cdots (2k-3) (2k-1)} = \Gamma\left(\frac{1}{2}\right) = \sqrt{\pi}.
	\end{displaymath}
	Dividing both sides by $\sqrt{2}$ and taking the square of both sides, we have
	\begin{displaymath}
		\lim_{k \to \infty} \frac{2}{1} \cdot \frac{2}{3} \cdot \frac{4}{3} \cdot \frac{4}{5} \cdots \frac{2k-2}{2k-3} \cdot \frac{2k-2}{2k-1} \cdot \frac{2k}{2k-1} = \frac{\pi}{2},
	\end{displaymath}
	which implies the desired formula.
\end{proof}

Combining \eqref{conv_even} and \eqref{conv_odd}, we can provide a proof of the following.

\begin{corollary}[duplication formula]
	For any $\alpha > 0$, we have
	\begin{displaymath}
		\Gamma(\alpha) \Gamma\left(\alpha + \frac{1}{2}\right) = 2^{1-2\alpha} \sqrt{\pi} \Gamma(2\alpha).
	\end{displaymath}
\end{corollary}
\begin{proof}
	Multiplying \eqref{conv_even} and \eqref{conv_odd}, we have
	\begin{displaymath}
		\sqrt{\pi}k^{2\alpha} \frac{\left(k-\frac{1}{2}\right)(k-1)\cdots1\cdot\frac{1}{2}}{\left(k-\frac{1}{2}+\alpha\right)(k-1+\alpha)\cdots\left(\alpha+\frac{1}{2}\right)\cdot\alpha}
		\to \Gamma(\alpha)\Gamma\left(\alpha+\frac{1}{2}\right).
	\end{displaymath}
	By multiplying $2^{2k}$ to both the numerator and the denominator, we have
	\begin{displaymath}
		2^{1-2\alpha}\sqrt{\pi}(2k)^{2\alpha}\frac{(2k-1)(2k-2)\cdots 1}{(2k-1+2\alpha)(2k-2+2\alpha) \cdots (2\alpha)}
		\to \Gamma(\alpha)\Gamma\left(\alpha+\frac{1}{2}\right).
	\end{displaymath}
	By noticing that the previous formula contains \eqref{conv_even} with $\alpha$ and $k$ replaced by $2\alpha$ and $2k$, we obtain the desired formula.
\end{proof}

\end{document}